\documentclass[12pt]{article}
\usepackage[OT2,OT1]{fontenc}
\usepackage{amsfonts,amsmath, amssymb,latexsym}
\usepackage[all,cmtip]{xy}
\usepackage{float}
\usepackage[stable]{footmisc}
\usepackage[title]{appendix}

\def\Z{{\mathbb Z}}

\def\SL{{\rm SL}}
\def\GL{{\rm GL}}
\def\SO{{\rm SO}}
\def\PGL{{\rm PGL}}

\def\inv{{\rm inv}}

\def\Sym{{\rm Sym}}

\def\Cl{{\rm Cl}}

\def\OO{{\mathcal O}}
\def\OO{{\mathcal O}}

\def\Disc{{\rm Disc}}
\def\Disc{{\rm Disc}}
\def\Det{{\rm Det}}
\def\det{{\rm det}}

\def\r{{\rm r}}

\def\Tr{{\rm Tr}}

\def\Q{{\mathbb Q}}

\def\Z{{\mathbb Z}}

\def\Q{{\mathbb Q}}

\def\cR{{\mathcal R}}

\def\wzn2{{W_{\Z,+}^{(2-)}}}

\def\fz1{{F_{\Z,1}}}

\newtheorem{theorem}{Theorem}[section]

\newtheorem{corollary}[theorem]{Corollary}

\newtheorem{definition}[theorem]{Definition}

\newenvironment{proof}{\noindent {\bf Proof:}}{$\Box$ \vspace{2 ex}}

\begin{document}

\title{A quartic order is at most 20,000 times monogenic} 
\title{On the number of monogenizations of a quartic order} 

\author{Manjul Bhargava \\ (with an Appendix by Shabnam Akhtari)}

\maketitle

\begin{abstract}
We show that an order in a quartic field has fewer than 3000 essentially different generators as a $\Z$-algebra (and fewer than 200 if the discriminant of the order is sufficiently large).  This significantly improves the previously best known bound of~$2^{72}$.

Analogously, we show that an order in a quartic field is isomorphic to the invariant order of at most 10 classes of integral binary quartic forms (and at most 7 if the discriminant is sufficiently large).  This significantly improves the previously best known bound of~$2^{80}$. 

\end{abstract}

\setcounter{tocdepth}{4}


\section{Introduction}

Let $\OO$ be an order in a number field $K$. The ring $\OO$ is said to be {\it monogenic} if it is generated by one element as a $\Z$-algebra, i.e., $\OO=\Z[\alpha]$ for some $\alpha\in\OO$; the element $\alpha$ is then called a {\it monogenizer} of $\OO$.  Note that if $\alpha$ is a monogenizer of $\OO$, than so is $\pm \alpha+c$ for any $c\in\Z$. Two monogenizers $\alpha$ and $\alpha'$ of $\OO$ are called {\it equivalent} if $\alpha'=\pm \alpha +c$ for some $c\in\Z$. An equivalence class of monogenizers of $\OO$ is called a {\it monogenization} of $\OO$.

It is a beautiful theorem of Gy\H{o}ry~\cite{G} that any order $\OO$ can have at most finitely many monogenizations.
This naturally raises the question as to how many monogenizations an order~$\OO$ can have.
It is elementary that an order~$\OO$ in a quadratic field $K$ always has exactly one monogenization.   For cubic fields, it follows from work of Bennett~\cite{B} on cubic Thue equations that an order~$\OO$ in a cubic field has at most 10 monogenizations.

For degrees $n\geq 4$, Evertse and Gy\H{o}ry proved in~\cite{EG} that an order~$\OO$ in a number field~$K$ of degree $n$ can have at most $(3\times 7^{2n!})^{n-2}$ monogenizations. The best known result to date for $n\geq 4$ is due to Evertse~\cite{E}, who proved that an order~$\OO$ in a number field~$K$ of degree $n$ can have at most $2^{4(n+5)(n-2)}$ monogenizations.
In the case $n=4$, Evertse's result thus shows that an order in a quartic field can have at most~$2^{72}$ monogenizations. 

The purpose of this article is to substantially improve this previously best known bound in the case $n=4$ as follows. 

\begin{theorem}\label{main1}
An order $\OO$ in a quartic number field can have at most $2760$
 monogenizations $($and at most $182$ monogenizations if $|\Disc(\OO)|$ is sufficiently large$)$. 
\end{theorem}

A related problem is to understand the number of ways an order $\OO$ in a number field~$K$ of degree~$n$ can arise as the {\it invariant order} of an integral binary $n$-ic form.  \pagebreak Given an irreducible integral binary $n$-ic form $f(x,y)=a_0x^n+a_1x^{n-1}y+\cdots+a_ny^n$, there is a naturally associated order $\OO_f$ in the degree-$n$ field $K_f=\Q[x]/f(x,1)$, given as follows.  If $\theta$ denotes a root of $f(x,1)$ in $K_f$, then the invariant order $\OO_f$ of $f$ has $\Z$-basis given by 
\[\mbox{1, \,$a_0\theta$, \, $a_0\theta^2+a_1\theta$, \;\,$\ldots$\,, \,\,$a_0\theta^{n-1}+a_1\theta^{n-2}+\cdots+a_{n-2}\theta$.}\]  If $f$ and $f'$ are integral binary $n$-ic forms that are in the same class (i.e., $f'(x,y)=f((x,y)\gamma)$ for some element $\gamma\in\GL_2(\Z)$), then $\OO_f$ and $\OO_{f'}$ are also isomorphic as rings. Moreover, $\Disc(\OO_f)=\Disc(f)$ for any integral binary $n$-ic form $f$. The ring $\OO_f$ can in fact be naturally defined even when $f$ is not irreducible or does not have leading coefficient nonzero; see~\cite{Nakagawa, Wood2} for further details and references. 

It follows from the seminal work of Birch and Merriman~\cite{BM} that an order~$\OO$ in a number field~$K$ of degree $n$ can arise as the invariant order 
of at most finitely many classes of binary $n$-ic forms.  This naturally raises the question as to how many classes of binary $n$-ic forms there can be having a given invariant order.
 In the case $n=3$, the parametrization of cubic orders due to Levi and Delone--Faddeev implies that every order $\OO$ in a cubic field is the invariant order of a {\it unique} integral binary cubic form up to equivalence.  For degrees $n\geq 4$, B\'erczes, Evertse, and Gy\H{o}ry proved in \cite{BEG} that an order $\OO$ in a number field $K$ of degree $n$ can be the invariant order of at most $2^{24n^3}$ classes of integral binary $n$-ic forms.  This upper bound was subsequently improved to $2^{5n^2}$ by Evertse and Gy\H{o}ry~\cite[Theorem~17.1.1]{EG2}. 

In the case $n=4$, the result of Evertse and Gy\H{o}ry shows that an order in a quartic field can arise as the invariant order of at most~$2^{80}$ classes of binary quartic forms. In this article, we substantially improve this previously best known bound in the case $n=4$ as~follows. 

\begin{theorem}\label{main2}
An order $\OO$ in a quartic number field is isomorphic to the invariant order of at most $10$ classes of integral binary quartic forms $($and at most $7$ classes if $|\Disc(\OO)|$ is sufficiently large$)$. 
\end{theorem}

We prove Theorem~\ref{main1} 
via a parametrization of quartic rings and their cubic resolvent rings by pairs of integral ternary quadratic forms that we established in~\cite{hcl3}, together with a parametrization of quartic rings having monogenized cubic resolvent rings by integral binary quartic forms due to Wood~\cite{Wood}.  The relationship between quartic index equations and pairs of ternary quadratic forms was first studied by Ga\'al, Peth\H{o} and Pohst in~\cite{GPP}. Here, our aim is to take this relationship further by interpreting this connection using the aforementioned parametrizations of quartic rings and their cubic resolvents. 

Specifically, we apply the above parametrizations of quartic rings and their cubic resolvent rings by suitable pairs of ternary quadratic forms to reduce Theorem~\ref{main1} to counting solutions to certain pairs of Thue equations, one of which is quartic and the other cubic.  
The bound in Theorem~\ref{main1} then arises as a product of the best known bounds on the number of solutions to quartic Thue equations and the number of solutions to cubic Thue equations.

The best known upper bound on the number of solutions to cubic Thue equations is due to Bennett~\cite{B}, and due to Okazaki~\cite{O} in the case of sufficiently large absolute discriminant (see Theorem~\ref{BO}).
Meanwhile, by combining her work in \cite{Akhsmallm, Akh10} with the recent computational results of Bennett and Rechnitzer~\cite{BR}, Akhtari deduces the best known upper bounds on the number of solutions to quartic Thue equations in Appendix A (see Theorem~\ref{SSthm}). We thank her for this valuable contribution. 

Meanwhile, the bounds in Theorem~\ref{main2} are seen to arise along the way as the best known bounds on the number of solutions to cubic Thue equations.  Thus any future improvements to the aforementioned best known bounds on the number of solutions to cubic and quartic Thue equations would also yield corresponding improvements to the bounds stated in Theorems~\ref{main1} and \ref{main2}.


\section{Preliminaries on the quadratic and cubic cases}

\subsection{Rings of rank $n$}

A {\it ring of rank $n$} is a commutative ring with unity that is free of rank $n$ as a $\Z$-module.  Rings of rank $n=2$, 3, 4, or 5 are called {\it quadratic}, {\it cubic}, {\it quartic}, and {\it quintic} rings, respectively.  

To any ring $\cR$ of rank $n$ we may attach
the {\it trace} function $\Tr:\cR\rightarrow\Z$, which assigns to an
element $\alpha\in \cR$ the trace of the endomorphism
\smash{$\cR\!\xrightarrow{\!\times \alpha\!}  \!\cR$}. The {\it discriminant}
$\Disc(\cR)$ of such a ring $\cR$ is then defined as the determinant
$\det(\Tr(\alpha_i \alpha_j))\in\Z$, where $\{\alpha_i\}$ is any
$\Z$-basis of~$\cR$.  

The {\it content} $c(\cR)$ of a ring $\cR$ of rank $n$ is the maximum integer $k$ such that $\cR=\Z+k\cR'$ for some ring $\cR'$ of rank $n$; the ring $\cR$ is called {\it primitive} if its content is $1$.  The content~$c(\OO_f)$ of the invariant order $\OO_f$ of a binary $n$-ic form $f$ is equal to the content~$c(f)$ of the form~$f$ itself (i.e., the gcd of the coefficients of $f$), and indeed $\OO_f=\Z+c(\OO_f)\OO_{f/c(f)}$. If $f$ has leading coefficient 1, then $f(x,1)$ is a monic integer polynomial in one variable, and $\OO_f=\Z[x]/(f(x,1))$; moreover, $\OO_f$ is primitive in this case, since the gcd of the coefficients of $f$ is 1. 
An explicit multiplication table for~$\OO_f$ in terms of the coefficients of $f$, for a general integral binary form $f$, was given by Nakagawa~\cite{Nakagawa}; see also the work of Wood~\cite{Wood2} for a coordinate-free description of~$\OO_f$. 

Important in the proof of Theorems~\ref{main1} and \ref{main2} is the parametrization of quartic rings by pairs of integral ternary quadratic forms.  Before discussing this parametrization, we review the parametrizations of  rings of rank $n\leq 3$. 
(The only ring of rank 1 is $\Z$.)  

\subsection{The quadratic case}

Quadratic rings are parametrized by their discriminants~$D$, which are integers congruent to 0 or 1 (mod~4).  The unique quadratic ring $S(D)$ of discriminant $D$, up to isomorphism, is given by $$S(D):=\Z[(D+\sqrt{D})/2],$$
or alternatively, $S(D):=\Z[x]/(x^2-D+(D^2-D)/4)$.  In particular, all quadratic rings are monogenic, and indeed all have exactly one monogenization.  

The invariant ring of an integral binary quadratic form of discriminant $D$ is simply~$S(D)$.  Thus the number of $\GL_2(\Z)$-classes of integral binary quadratic forms having given invariant ring $S(D)$ is equal to the number of such classes having discriminant $D$.  

The latter number of classes is also equal to the number of (not-necessarily-invertible) oriented ideal classes $I$ of $S(D)$, up to automorphisms of $S(D)$. (See, e.g.,~\cite[\S3.2]{hcl1} for the definition of an oriented  ideal class.) This can be seen by noting that an oriented ideal class $I$ of $S(D)$ naturally corresponds to the integral binary quadratic form given by the norm form of~$I$ (i.e., if $I$ is spanned by $\alpha,\beta$ as a $\Z$-module, then $I$ corresponds to the binary quadratic form $N(\alpha x+ \beta y)/N(I)$).  Conversely, it is well known that an integral binary quadratic form $f$ of discriminant $D$ arises in this way from a unique oriented ideal class $I$ of $S(D)$, up to automorphisms of $S(D)$, whose norm form is in the $\GL_2(\Z)$-class of $f$ (see, e.g.,~\cite[\S3.2]{hcl1}).

The number of oriented ideal classes $I$ of $S(D)$, up to automorphisms of $S(D)$, can in turn be described in terms of the {narrow class group} of~$S(D)$ (in the sense of~\cite[\S3.2]{hcl1}) and of its subrings. Namely, the {invertible} oriented ideal classes $I$ of~$S(D)$, up to automorphisms of~$S(D)$, are in bijection with the elements $I$ of the oriented class group $\Cl^+(S(D))$ of~$S(D)$, modulo~the equivalence relation $I\sim I^{-1}$;  
the number of such classes is thus given by $\frac12[h^+(S(D))+h_2^+(S(D))]$, where $h^+(S(D)):=\#\Cl^+(S(D))$ and $h_2^+(S(D)):=\#\Cl^+(S(D))[2]$. 
We conclude that the number of (not-necessarily-invertible) ideal classes $I$ of~$S(D)$, up to automorphisms of~$S(D)$, is equal to 
\begin{equation}\label{cnformula}
\sum_{\scriptstyle{k^2|D}\atop \scriptstyle{D/k^2\,\equiv\, 0 \mathrm{\:or\:}1\!\!\!\!\!\pmod 4}} \textstyle\frac12[h^+(S(D/k^2))+h_2^+(S(D/k^2))].
\end{equation}

We summarize this discussion in the following theorem:

\begin{theorem}
Every quadratic ring $S(D)$ has exactly one monogenization.  The number of classes of binary quadratic forms having invariant order isomorphic to $S(D)$ is 
the number 
of 
\:\!$($\:\!\!not-necessarily-invertible$)\!$ 
oriented ideal classes $I$ of $S(D)$, up to automorphisms of~$S(D)$,~and is given by $(\ref{cnformula})$.
\end{theorem}

\subsection{The cubic case}

The parametrization of cubic rings is due to Levi~\cite{Levi}, Delone--Faddeev~\cite{DF}, and in its most general form, Gan--Gross--Savin~\cite{GGS}.  The statement is that isomorphism classes of cubic rings are in natural one-to-one-correspondence with classes of integral binary cubic forms.  Given an integral binary cubic form $f(x,y)=ax^3+bx^2 y+cx y^2+dy^3$, the associated cubic ring $R(f)$ has $\Z$-basis $\langle1$, $\omega$, $\theta\rangle$, with multiplication table given by 
\begin{equation}\label{ringlaw3}
\begin{array}{cll}
  \omega\theta &=& -ad \\
  \omega^2 &=& -ac - b \omega + a \theta \\
  \theta^2 &=& -bd - d \omega + c \theta.
\end{array}
\end{equation}
This is indeed the multiplication table for the invariant order of the binary cubic form $f$. 

Conversely, given a cubic ring $R$, let $\langle 1,\omega,\theta\rangle$ be a
  $\Z$-basis for $R$.  Translating $\omega$ and $\theta$ by the appropriate
  elements of $\Z$, we may assume that $\omega\theta\in\Z$; in
  the terminology of~\cite{DF}, a basis satisfying the latter
  condition is called {\it normal}.
If $\langle 1,\omega,\theta\rangle$ is a normal basis, then there exist constants
$a,b,c,d,\ell,m,n\in\Z$ such that
\begin{equation}\label{ringlaw3b}
\begin{array}{cll}
  \omega\theta &=& n \\
  \omega^2 &=& m - b \omega + a \theta \\
  \theta^2 &=& \ell\, - d \omega + c \theta.
\end{array}
\end{equation}
One checks that the associative law $\omega^2\cdot\theta=\omega\cdot \omega\theta$ and $\omega\cdot\theta^2=\omega\theta\cdot \theta$ implies that $n=-ad$, $m=-ac,$, and $\ell=-bd$.  Thus, to the cubic ring $R$, we may associate the binary cubic form $f(x,y)=ax^3+bx^2 y+cx y^2+dy^3$, and we clearly have $R=R(f)$.  

One checks that changing the $\Z$-basis $\langle\omega,\theta\rangle$ of $R/\Z$ by an element $\gamma\in\GL_2(\Z)$, and then renormalizing, results in $f(x,y)$ being replaced by $\det(\gamma)^{-1}f((x,y)\gamma)$; 
thus changing the basis of $R(f)/\Z$ leads to a $\GL_2(\Z)$-equivalent form under this correspondence. 

We have proven the following theorem.

\begin{theorem}
There is a bijection between $\GL_2(\Z)$-classes of integral binary cubic forms and isomorphism classes of cubic rings.  
\end{theorem}

In coordinate-free terms, the form $f(x,y)$ represents the cubic map
$R/\Z\to \wedge^2(R/\Z)\cong\Z$ given by $r\mapsto r\wedge r^2$.  
There is again an equality of discriminants $\Disc(R(f))=\Disc(f)$, and the cubic rings $R(f)$ that are integral domains correspond to binary cubic forms~$f$ that are irreducible over $\Q$.  For details, see~\cite[\S2]{BST}. 

We may now ask: for which integral binary cubic forms $f(x,y)$ is the corresponding cubic ring $R(f)$ monogenic?  To answer this question, we observe that, if $R(f)$ is monogenic, then via a change of basis, we may assume that $\omega$ is a generator of $R(f)$ as a $\Z$-algebra.  The determinant of the change of basis of $R(f)$ from $\langle 1,\omega,\omega^2\rangle$ to $\langle1,\omega,\theta\rangle$ is clearly seen from (\ref{ringlaw3}) to be $a$.  Thus, by switching the sign of $\omega$ if necessary, we must have $a=1$.  
Conversely, if $a=1$, then $R(f)$ is spanned over $\Z$ by $1,\omega,\omega^2$, and so is generated by $\omega$ as a $\Z$-algebra.

It follows that $R(f)$ is monogenic if and only if $f(x,y)$ represents 1; moreover, since the action of $\GL_2(\Z)$  on the basis $\langle\omega,\theta\rangle$ of $R/\Z$ results in a corresponding action of $\GL_2(\Z)$ on the form $f$, we see that the monogenizations of $R(f)$ are naturally in one-to-one correspondence with the representations of 1 by~$f$. 

\begin{theorem}
Let $R$ be a cubic ring and $f$ its associated integral binary cubic form. Then there is a bijection between the monogenizations of $R$ and the representations of $1$ by $f$.
\end{theorem}

Thus, to bound the number of monogenizations of a cubic order~$R(f)$, we are reduced to bounding the number of solutions of the cubic Thue equation $f(x,y)=1$.  This equation has a long history, going back to the work of Thue~\cite{Thue}, who proved the finiteness of the number of solutions for binary forms $f$ of any degree, with subsequent improved bounds and variants by many authors including Siegel~\cite{Siegel}, Lewis--Mahler~\cite{LM}, Evertse~\cite{Eve}, Bombieri--Schmidt~\cite{BoSch}, Stewart~\cite{Ste}, and Gy\H{o}ry~\cite{Gyo2001} (see also~\cite{Akhtr,YuZh, SS, Kre39}).  

The best known bounds currently on the number of solutions to cubic Thue equations are due to Bennett~\cite{B} and Okazaki~\cite{O} (see also Akhtari~\cite{A}), in the cases of arbitrary nonzero discriminant and sufficiently large absolute discriminant, respectively:

\begin{theorem}[Bennett~\cite{B}, Okazaki~\cite{O}]\label{BO}
An integral binary cubic form $f$ of nonzero discriminant can represent $1$ at most $10$ times, and at at most $7$ times if $f$ has sufficiently large absolute discriminant. 
\end{theorem}
Hence a cubic ring of nonzero discriminant can have at most 10 monogenizations, and a cubic ring of sufficiently large absolute discriminant can have at most 7 monogenizations.  (It is known that $\Z[\zeta_7+\zeta_7^{-1}]$ has 9 monogenizations, so the constant 10 in Theorem~\ref{BO} is thus optimal or nearly optimal!) 
We summarize this discussion in the following theorem. 

\begin{theorem}\label{cubmon}
A cubic ring $R$ can have at most $10$ monogenizations, and a cubic ring of sufficiently large absolute discriminant at most $7$ monogenizations.  Every cubic ring is isomorphic to the invariant order of exactly one class of binary cubic forms.
\end{theorem}

\section{The quartic case}

\subsection{The parametrization of quartic rings and cubic resolvent rings}

We now recall the parametrization of quartic rings together with their cubic resolvent rings, by pairs of integral ternary quadratic forms, as developed in \cite{hcl3}. To state the theorem, 
let~$W_\Z$ denote the space of pairs $(A,B)$ of ternary quadratic forms
having coefficients in $\Z$. We identify ternary quadratic
forms over $\Z$ with their Gram matrices whose coefficients lie in~$\frac12\Z$; we may thus express an element $(A,B)\in W_\Z$ as a pair
of $3\times 3$ symmetric matrices via
$$2 \cdot (A,B)=\left(\left[\begin{array}{ccc} 2a_{11} & a_{12} & a_{13} \\ a_{12} & 2a_{22} & a_{23} 
\\ a_{13} & a_{23} & 2a_{33} \end{array} \right],
  \left[ \begin{array}{ccc} 2b_{11} & b_{12} & b_{13} \\ b_{12} &
      2b_{22} & b_{23} \\ b_{13} & b_{23} & 2b_{33} \end{array}
  \right] \right),$$ where $a_{ij},b_{ij}\in \Z$. 

The group
$\GL_3(\Z)\times \GL_2(\Z)$ acts naturally on the space $W_\Z$.
Namely, an element $g_3\in \SL_3(\Z)$ acts on $W_\Z$ by
$g_3\cdot(A,B)=(g_3Ag_3^t,g_3Bg_3^t)$, while an element
$g_2=\left(\begin{smallmatrix}{} p & q\\r & s \end{smallmatrix}\right)
\in \GL_2(\Z)$ acts by $g_2\cdot(A,B)=(pA+qB,rA+sB)$. The ring of
polynomial invariants for the action of $\GL_3(\Z)\times\GL_2(\Z)$ on
$W_\Z$ is generated by one element, which is called the
{\it discriminant}. The discriminant $\Delta(A,B)$ of an element $(A,B)\in W_\Z$ is
given by the discriminant of the binary cubic form $4\,\Det(Ax-By)$ in $x$ and
$y$, and is thus an invariant of degree 12 in the entries of $A$ and $B$.

We find that $\GL_3(\Z)\times
\GL_2(\Z)$-orbits on the space $W_\Z$ parametrize quartic rings together with their ``cubic resolvent rings'':

\begin{theorem}[\cite{hcl3}]\label{first}
There is a canonical bijection between the set of $\GL_3(\Z)\times
\GL_2(\Z)$-orbits on the space $W_\Z=(\Sym^2\Z^3\otimes\Z^2)^\ast$ of pairs
of integral ternary quadratic forms and the set of isomorphism classes
of pairs $(Q,R)$, where $Q$ is a quartic ring and $R$ is a cubic
resolvent ring of $Q$.  

Under this bijection, the discriminant of an
element $(A,B)\in(\Sym^2\Z^3\otimes\Z^2)^\ast$ is equal to the
discriminant of the corresponding quartic ring $Q$. Furthermore, the binary cubic form
corresponding to the cubic ring $R$ by the Delone--Faddeev
correspondence is $4\cdot\Det(Ax-By)$.
\end{theorem}
We denote the quartic ring corresponding to a pair of integral ternary quadratic forms by $Q(A,B)$, and the associated cubic resolvent ring by $R(A,B)$. 

We recall from \cite{hcl3} that a {\it cubic resolvent ring} $R$ of a quartic ring $Q$ is an integral analogue of the classical cubic resolvent field of a quartic field.  In the case of an order in an $S_4$- or $A_4$-quartic field $K$, it may be defined as a subring of the {cubic invariant ring } $R^\inv(Q)$ of $Q$, which is an order in the cubic resolvent field of $K$.  More precisely, let $\widetilde K$ denote the Galois closure of the quartic field $K$ in a fixed algebraic closure $\bar Q$ of $\Q$, let $K_1=K,$ $K_2$, $K_3$, $K_4$ denote the conjugates of $K$ in $\widetilde K$ permuted by the Galois group $S_4$, and let $F$ denote the cubic subfield of $\widetilde K$ fixed by $D_4=\langle (12),(1324)\rangle$. Then there is a natural map
\begin{equation}\label{phidef}
\tilde\phi(x) = x_1 x_2 + x_3 x_4
\end{equation}
from $K$ to $F$, used in the classical solution to the quartic
equation, where we have used $x_1,x_2,x_3,x_4$ to denote the conjugates
of $x$ in $\widetilde K$.  

The {\it cubic invariant ring} $R^\inv(Q)$ is defined as $\Z[\{x_1x_2+x_3x_4:x\in Q\}]$, i.e.,  the order in $F$ generated by $x_1x_2+x_3x_4$ over all $x\in Q$. Even when $Q$ is an order in a quartic field that is not $A_4$- or $S_4$-, we may still define $R^\inv(Q)$ as a cubic ring in a fixed formal Galois closure of $K$; see \cite{hcl3}. 

A cubic resolvent ring $R$ of the quartic ring $Q$ is then defined as follows. 

\begin{definition}
{\em Let $Q$ be a quartic ring, and $R^{inv}(Q)$ its cubic invariant ring.  
A {\it cubic resolvent ring} of $Q$ is 
a cubic ring $R$ such that $\Disc(Q)=\Disc(R)$ and 
$R\supset R^{inv}(Q)$.}
\end{definition}
We thus have a natural quadratic map $\tilde\psi(x):Q\to R$ for any quartic ring $Q$ and cubic resolvent ring $R$, defined by (\ref{phidef}), which evidently descends to a quadratic map $\psi(x):Q/\Z\to R/\Z$, since $\psi(x+c)=\psi(x)+c\Tr(x)+2c^2$ for any $x\in Q$ and $c\in\Z$. 

In coordinate-free language, Theorem \ref{first} states that isomorphism
classes of pairs $(Q,R)$---where $Q$ is a quartic ring and $R$ is a cubic 
resolvent---are in natural bijection with isomorphism classes of quadratic maps
$\phi:M\rightarrow L,$
where $M$ and $L$ are free $\Z$-modules having
ranks 3 and 2 respectively.  Under this bijection we have 
that $M=Q/\Z$ and $L=R/\Z$.  



Important to us is the following theorem on the existence of cubic resolvent rings, and the uniqueness of cubic resolvent rings in the primitive case:

\begin{theorem}[\cite{hcl3}]\label{monqr}
Every quartic ring has at least
one cubic resolvent ring.  In general, a quartic ring of content $k$ has $\sigma(k)$ cubic resolvent rings, where $\sigma$ denotes the sum-of-divisors function. Hence every primitive quartic ring has a unique
cubic resolvent ring.

In particular, every monogenic quartic ring $Q=\Z[x]/(x^4+bx^3+cx^2+dy+e)$ has a unique cubic resolvent ring $R$, and it is given by $R=\Z[x]/(x^3-cx^2+(bd-4e)x+4ce-d^2-b^2e)$. 
\end{theorem}

\subsection{The parametrization of quartic rings with monogenic cubic resolvent rings}

We now turn to the parametrization of quartic rings having monogenic cubic resolvent rings, as developed by Wood~\cite{Wood}.  Suppose $R=R(A,B)$ is a monogenic cubic ring, where $(A,B)$ is a pair of integral ternary quadratic forms. If $f(x,y)=ax^3+bx^2y+cxy^2+dy^3$ denotes the integral binary cubic form $4\cdot\Det(Ax-By)$, then the fact that $R=R(f)$ is monogenic implies that $f$ represents 1.  By a change of variables in $\GL_2(\Z)$, we may thus assume $a=1$, i.e., $4\cdot\Det(A)=1$. 

Now if $4\cdot\Det(A)=1$, then $A$ is $\GL_3(\Z)$ equivalent to the matrix 
\begin{equation}\label{splitA} A_1:=
\left[ \begin{array}{ccc}\phantom0 & \phantom0 & 1/2
 \\ \phantom0 & -1 & \phantom0 \\ 1/2 & \phantom0 & \phantom0 \end{array} \right];
 \end{equation}
 see~\cite{Wood}. Therefore, if $R(A,B)$ is a monogenic cubic ring with multiplication table given by~(\ref{ringlaw3}) and monogenizer $\omega$, so that $f(1,0)=1$, then by a transformation in $\GL_2(\Z)\times
\GL_3(\Z)$, we may assume that $A$ is the quadratic form $A_1$ given by (\ref{splitA}). The quadratic form $B$ is then determined up to transformations in $\SO_{A_1}(\Z)$ (the orthogonal group of $A_1$, so that $A$ remains fixed as $A_1$), in $\{\pm1\}$ (which acts by either fixing or negating $B$), and in $F_{\Z,1}$, the group of unipotent lower triangular transformations over $\Z$ (which adds multiples of $A$ to $B$).  By a transformation in $F_{\Z,1}$, we may assume that the upper right entry of $B$ is zero.

Thus we obtain a representation of $\{\pm1\}\times \SO_{A_1}(\Z)$ on integral ternary quadratic forms $B$ whose top right entry is zero. In particular, we have a five-dimensional representation of $\SO_{A_1}(\Z)$.  Now the quadratic form in (\ref{splitA}) is $A_1(p,q,r)=q^2-pr$, which is a scalar multiple of the discriminant of the binary quadratic form $H(x,y)=px^2+2qxy+ry^2$.  Thus $\SO_{A_1}(\Z)$ is isomorphic to $\PGL_2(\Z)$, and our five-dimensional representation, being irreducible, must be isomorphic to the space of binary quartic forms when viewed as a representation of $\PGL_2$.

We can see this isomorphism explicitly, as follows.  Let $V_\Z$ denote the space of integral binary quartic forms.  Then $V_\Z$ embeds into $W_\Z$ 
via the map $\psi$ defined by
\begin{equation}\label{vtow}
\psi: ax^4+bx^3y+cx^2y^2+dxy^3+ey^4\mapsto \left( \left[ \begin{array}{ccc}\phantom0 & \phantom0 & 1/2
 \\ \phantom0 & -1 & \phantom0 \\ 1/2 & \phantom0 & \phantom0 \end{array} \right],
\left[ \begin{array}{ccc} a & b/2 & 0 \\ b/2 & c & d/2 \\ 0 & d/2 & e \end{array} \right]\right).
\end{equation}
We denote 
the set of all pairs $(A_1,B)$ of ternary quadratic forms in $W_\Z$ 
by $W_{\Z,1}$. The group
$F_{\Z,1}\times\{\pm1\}\times \SO_{A_1}\subset\GL_2(\Z)\times\SL_3(\Z)$ preserves
$W_{\Z,1}$.
We also note that the
map $\psi$ is {\it discriminant preserving}, i.e., the discriminant of
an element of $V_\Z$ is equal to the discriminant of its image in
$W_\Z$.  For a binary quartic form $f$, if we write
$\psi(f)=(A_1,B)$, then we
call the binary form $\Det(A_1x-By)$ the 
{\it cubic resolvent form} of $f$; note that this form is {\it monic},
i.e., its leading coefficient as a polynomial in $x$ is 1. 

Next, we observe that every 
$F_{\Z,1}$-equivalence class of $W_{\Z,1}$ contains a unique element
$(A_1,B)$ such that the top right entry of $B$ is equal to $0$.
It follows that $\psi$ maps the space of binary quartic forms $V_\Z$ bijectively to
the set of $F_{\Z,1}$-orbits on $W_{\Z,1}$ via the
composite map
$$V_\Z\rightarrow W_{\Z,1}\rightarrow F_{\Z,1}\backslash W_{\Z,1}.$$

We may ask how the action of $\GL_2(\Z)$ on $V_\Z$ explicitly manifests itself (via $\psi$) as an action on $F_{\Z,1}\backslash W_{\Z,1}$.
To answer this, note that the center of $\GL_2(\Z)$ acts trivially on its representation on
binary quadratic forms $px^2-2qxy+ry^2$ via $\gamma\cdot
f(x,y):=f((x,y)\cdot\gamma)/(\det \gamma)$.
This action of $\GL_2(\Z)$ preserves the discriminants $4(q^2-pr)$ of these binary quadratic forms, yielding the map
\begin{equation}\label{eqtwistedaction1}
\begin{array}{rcl}
\rho:\PGL_2(\Z)&\rightarrow&\SL_3(\Z),\;\; \mbox{given explicitly by}\\[.1in]
{\left(\begin{array}{cc} a & {b} \\ c & d \end{array}\right)}&\mapsto& \displaystyle\frac{{1}}{ad-bc}{
\left(\begin{array}{ccc} {d^2} & {cd} & {c^2} \\ {2bd} & {ad+bc} & {2ac} \\ {b^2} & {ab} & {a^2}
\end{array}\right).}
\end{array}
\end{equation}
Since $A_1$ is the Gram matrix of the ternary form
$q^2-pr$, we see that
the image of $\PGL_2(\Z)$ is contained in the orthogonal group
$\SO_{A_1}(\Z)$, and is in fact equal to it (see \cite{Wood}).

For any base ring $T$, let $V_T$ denote the space of binary quartic forms
with coefficients in $R$. The center of $\GL_2(T)$ acts trivially under
the ``twisted action'' of $\GL_2(T)$ on $V_T$ defined by 
\begin{equation}\label{eqtwistedaction}
\gamma\cdot f(x,y):=(\det\gamma)^{-2}f((x,y)\gamma),  
\end{equation}
yielding an action of $\PGL_2(R)$ on $V_T$. Note that the
$\PGL_2(\Z)$-orbits on $V_\Z$ are the same as the $\GL_2(\Z)$-orbits
on $V_\Z$, since
$\left(\begin{smallmatrix}{-1}&{}\\{}&{-1}\end{smallmatrix}\right)\in\GL_2(\Z)$
acts trivially on $V_\Z$.

It is now easily checked that $\psi(\gamma\cdot f)$ and
$\rho(\gamma)\cdot\psi(f)$ are the same element in $F_{\Z,1}\backslash
W_{\Z,1}$ for all $\gamma\in\PGL_2(\Z)$ and $f\in V_\Z$. Therefore,
we have the following theorem. 

\begin{theorem}\label{thmvztowz}
  The map $\psi$ defined by $(\ref{vtow})$ gives a canonical bijection
  between $\PGL_2(\Z)$-orbits on $V_\Z$ and $F_{\Z,1}\times {\rm
    SO}(A_1,\Z)$-orbits on $W_{\Z,1}$.
\end{theorem}

Finally, one checks that, for an integral binary quartic form $f$, the invariant order of the form $f$ is indeed isomorphic to the quartic ring $Q(\psi(f))$ associated to the pair of ternary quadratic forms $\psi(f)$.  Moreover, negating $B$ in a pair $(A_1,B)$ negates the monogenizer $\beta$ of the corresponding cubic resolvent ring $R$ in the associated pair $(Q,R)$; and negating $B$ corresponds to negating $f$ in the corresponding binary quartic form.  Since we are considering $\beta$ and $-\beta$ as equivalent monogenizers of $R=\Z[\beta]$, we obtain the following theorem due to Wood (see~\cite{Wood} for more details):

\begin{theorem}[Wood~\cite{Wood}]\label{woodthm}
The orbits of $\{\pm1\}\times \PGL_2(\Z)$ on the space $V_\Z$ of integral binary quartic forms $f$ are in natural bijection with
isomorphism classes of triples $(Q,R,\beta)$, where $Q$ is a quartic ring, $R$ is a monogenic cubic resolvent ring 
of $Q$, and $\beta$ is a {\it monogenizer} of~$R$ up to equivalence. Under this bijection, the invariant order of a binary quartic form $f$ is isomorphic to the quartic ring $Q$ in the corresponding triple $(Q,R,\beta)$.
\end{theorem}

\subsection{Conclusion: Proofs of Theorems~\ref{main1} and \ref{main2}}

We begin by bounding the number of classes of binary quartic forms that can yield a given quartic order $Q$ as its invariant order.  

First, we may reduce to the case of primitive quartic orders and primitive binary quartic forms.  Indeed, we have $Q=\OO_f$ if and only if $Q'=\OO_{f/c(f)}$, where $Q'$ is the primitive quartic ring satisfying $Q=\Z+c(Q)Q'$.
Hence the number of classes of $f$ such that $Q=\OO_f$ is equal to the number of classes of primitive forms $f'$ such that $Q'=\OO_{f'}$.  

Second, by Theorem~\ref{woodthm}, given a primitive quartic order $Q'$, the number of classes of integral binary quartic forms, up to negation, whose invariant order is isomorphic to $Q'$ is given by the number of pairs $(R',\beta)$, where $R'$ is a cubic resolvent ring of $Q'$ and $\beta$ is a monogenizer of~$R'$.  By Theorem~\ref{monqr}, such a cubic resolvent $R'$ is unique since $Q'$ is primitive.  
Furthermore, by Theorem~\ref{cubmon}, the number of possible monogenizers $\beta$ of $R'$ is at most 10 (and at most 7 if $|\Disc(Q)|=|\Disc(R)|$ is sufficiently large).  We have proven Theorem~\ref{main2}. 

\vspace{.1in}
We now consider the number of possible monogenizers of a quartic order $Q$.   Suppose $Q= \Z[\alpha]$ for some monogenizer $\alpha\in Q$.  Let $g(x)$ be the (monic integral) characteristic polynomial of the map $Q\!\xrightarrow{\!\times \alpha\!}  \!Q$, so that $Q\cong \Z[x]/(g(x))$, and let $h(x,y)$ be the binary quartic form that is the homogenization of $g$. If $R=R(f)$ is the (unique) cubic resolvent ring of $R$ with corresponding monogenizer $\beta$ as in Theorem~\ref{monqr}, then the triple $(Q,R,\beta)$ corresponds to the $\PGL_2(\Z)$-orbit of the integral binary quartic form $h(x,y)$ under the bijection of Theorem~\ref{woodthm}. 
The number of distinct monogenizers that lead to the same triple $(Q,R,\beta)$, via the homogenization of its characteristic polynomial, is then clearly the number of representations of $\pm1$ by $h$ (i.e., the number of monic $g(x)$, up to transformations $x\mapsto \pm x+c$ with $c\in\Z$, whose homogenization is $\PGL_2(\Z)$-equivalent to $h(x,y)$). 

On the other hand, we have seen above in the proof of Theorem~\ref{main2} that the number of different triples $(Q,R,\beta)$ for a given monogenic quartic order $Q$ is at most the number of monogenizations of its cubic resolvent ring $R=R(f)$, i.e., the number of representations of~1 by the binary cubic form $f$ corresponding to $R$.  

Thus the number of monogenizations of $Q$ is equal to the number of representations of~$\pm1$ by the binary quartic form $h$ corresponding to the triple $(Q,R,\beta)$, summed over all possible monogenizers $\beta$ of the cubic resolvent ring $R=R(f)$ of $Q$, which are in turn indexed by the representations of $1$ by the binary cubic form~$f$!  

In particular, we may bound the number of possible monogenizations of a quartic order $Q$ by the product of the maximal possible number of solutions to a quartic Thue equation $h(x,y)=\pm1$ (where the solutions $(x,y)$ and $(-x,-y)$ are considered equivalent) and the maximal possible number of solutions to a cubic Thue equation $f(x,y)=1$.  

The best known current bounds are due to Akhtari, and are proven in the Appendix:

\begin{theorem}[Akhtari]\label{SSthm}
An integral binary quartic form $f$ of nonzero discriminant represents $\pm1$ at most $276$ times, and at most $26$ times if $f$ has sufficiently large absolute discriminant. 
\end{theorem}
Noting that $276\times 10=2760$ and $26\times 7=182$ thus yields Theorem~\ref{main1}.

\subsection*{Acknowledgements}

I am extremely grateful to Shabnam Akhtari, Levent Alpoge, K\'alm\'an Gy\H{o}ry, Attila Peth\H{o}, and Ari Shnidman for helpful discussions and comments on earlier versions of this manuscript.  I also thank Shabnam Akhtari for the beautiful results proven in the Appendix below on the number of solutions to quartic Thue equations. The author was partially supported by a Simons Investigator Grant and NSF Grant DMS-1001828.

\begin{appendices}

\section{(by Shabnam Akhtari\footnote{Supported by NSF Grant DMS-2001281.})}

The purpose of this Appendix is to combine results from \cite{Akhsmallm, Akh10, BR} to conclude the following.
\begin{theorem}\label{main26}
A quartic Thue equation $f(x , y) = \pm 1$ has at most $276$ integer solutions and, moreover, if the absolute discriminant of $f(x , y)$ is sufficiently large, then $f(x , y) = \pm 1$ has at most $26$ integer solutions, where $(x,y)$ and $(-x,-y)$ are considered the same solution.
\end{theorem}

The bound $276$ in Theorem \ref{main26} is obtained using the computational results of Bennett and Rechnitzer in their work in progress \cite{BR}. 
They show that if $|D| \leq 10^6$, where $D=\Disc(f)$, then the equation $f(x , y) = \pm 1$  has at most $8$ integer solutions. We thank Mike Bennett for sharing their results with us.

It is conjectured that the number of integer solutions to quartic Thue equations $f(x,y)=\pm1$ does not exceed $8$. 
This conjecture may be verified for forms with absolute discriminant larger than $10^6$ in the future by computational methods. We will list how such computations would improve the upper bounds in Theorem~\ref{main26} in a table at the end.

The upper bound $26$ of Theorem~\ref{main26} was established in \cite{Akh10},  assuming that the discriminant $D$ of the quartic form $f(x, y)$ satisfies  $|D| \geq D_0$ where $D_0$ is an effectively computable constant. 
Essentially no effort was made in \cite{Akh10} to make the constant $D_0$ explicit, as the main goal was to establish the best possible upper bound for the number of solutions of quartic Thue equations by excluding only a finite number of $\textrm{GL}_{2}(\mathbb{Z})$-equivalence classes of quartic forms.  From a less technical part of this work, namely \cite[Section~4]{Akh10},  it is clear  that $D_0$ must be taken at least $4^4 (3.5)^{4\times 6 \times 65}$. From more complicated parts,  where the theory of linear forms in logarithms is applied, the existence of such an effectively computable constant is clear, but no estimates can be easily  provided. 

In  \cite{Akhsmallm}, upper bounds for the number of primitive integer solutions of Thue equations $|f(x , y)| = m$, having any given degree $n \geq 3$, were obtained under the assumption that the positive integer $m$ is sufficiently small in terms of the absolute discriminant of the binary form $f(x , y)$. The following is \cite[Theorem~1.2]{Akhsmallm} stated for  $m=1$ and  $n = 4$. 

\begin{theorem}\label{maineq}
 Let $f(x , y) \in \mathbb{Z}[x , y]$ be an irreducible binary quartic form with discriminant~$D$.  Let
$D_1 = (3.5)^{12} 4^{ 4}\approx 8.651\times 10^8$, and 
 assume
 \smash{$|D| > D_1^{{1}/{(1-\kappa)}},$}
 where $0 < \kappa < 1$.
 Then the number of integer solutions to the equation $|f(x , y)| = 1$ is at most
\begin{equation}\label{numnberinkappa}
 \lfloor36 + {8}/{ \kappa}\rfloor.
 \end{equation}
 \end{theorem}

To simplify some of the expressions in Theorem~\ref{maineq},  we have replaced $0 < \epsilon < {1}/{6}$ in the statement of \cite[Theorem~1.2]{Akhsmallm} by ${\kappa}/{6}$, where $0 < \kappa < 1$.

\begin{corollary}\label{corofthm}
Let $f(x,  y) \in \mathbb{Z}$ be a binary quartic form with discriminant $D$. If $|D| \geq  2.71336712 \times 10^{80}$, then the equation $f(x , y) = \pm 1$ has at most $44$ integer solutions.
\end{corollary}
\begin{proof}
Since for any $0< \kappa < 1$, we have  $\lfloor36 +  {8}/{\kappa} \rfloor\geq  44$, 
we choose a value of $\kappa$ for which $36 + {8}/{\kappa} < 45$, i.e.,   
$
\kappa > {8}/{9}.
$
Taking $\kappa = 0.888888889 > {8}/{9}$,
 we have 
$$
D_1^{{1}/{(1 -\kappa)}} = 2.71336712 \times 10^{80}.
$$
By Theorem \ref{maineq},  we conclude that $f(x , y) = \pm 1$ has at most $44$ integer solutions.
\end{proof}

The upper bounds in Theorem~\ref{maineq} are not quite as small as the bound $26$ in Theorem~\ref{main26},
but they can be applied to deduce bounds for the number of solutions of quartic equations $f(x , y) = \pm 1$ with smaller discriminant and with given discriminant ranges as well.  


We will deduce Theorem~\ref{main26} from the following more general version of Theorem~\ref{maineq}.

\begin{theorem}\label{gencount}
 Let $f(x , y) \in \mathbb{Z}[x , y]$ be an irreducible binary quartic form with discriminant~$D$.  Let
$D_1 = (3.5)^{12} 4^{ 4}\approx 8.651\times 10^8$, and 
let $r$ be a positive integer such that 
 \smash{$|r^{12}D| > D_1^{{1}/{(1-\kappa)}},$}
 where $0 < \kappa < 1$.
 Then the number of integer solutions to the equation $|f(x , y)| = 1$ is at most
\begin{equation}\label{numnberinkappa2}
 \psi(r)\lfloor36 + {8}/{ \kappa}\rfloor,
\end{equation}
 where $\psi$ is the Dedekind $\psi$-function defined by
 $$\psi(r)=r\prod_{p\mid r} (1+1/p).$$
 
\end{theorem}

\begin{proof}
Let $f(x , y)$ be a binary form of degree $n$ and let 
$A=\bigl(\begin{smallmatrix}a & b\\ c&d\end{smallmatrix}\bigr)$
be a $2 \times 2$ integral  matrix. Define the binary form $f_{A}$ by
$$
f_{A}(x , y) := f(ax + by \ ,\  cx + dy).$$
Then
\begin{equation}\label{St6}
\Disc(f_{A}) = (\det\, A)^{n (n-1)} \Disc(f).
\end{equation}

Now let $r$ be a positive integer.  There exist $\psi(r)$ sublattices $L_1,L_2,\ldots,L_{\psi(r)}\subset \Z^2$ such that $\Z^2/L_i\cong \Z/r\Z$. Let $A_1,\ldots,A_{\psi(r)}$ be integer $2\times2$ matrices such that $A_i$ cor\-responds to a linear transformation mapping $\Z^2$ onto $L_i$; thus $|\det(A_i)|=r$ for each~$i$.

We have $\Z^2=\cup_{i=1}^{\psi(r)}L_i$. 
 Therefore, the number of solutions of $|f(x , y)|=  1$ is at most 
 $$N_{f_{1}} + N_{f_{2}} + \cdots + N_{f_{\psi(r)}},$$
  where
$N_{f_j}$ denotes the number of solutions to $f_{A_j} (x , y)= \pm1$.
By (\ref{St6}), we have
$$
\left| \Disc(f_{A_{j}})\right| = r^{n(n-1)} |\Disc(f)|.
$$
 This means that if $N$ is an upper bound for the number of solutions to 
 $g(x , y) = \pm 1,$
  where~$g$ ranges over all irreducible binary quartic forms 
  with absolute discriminant at least $C\, r^{n(n-1)}$
for some positive constant $C$,  then  
\begin{equation}\label{p+1N}
\psi(r)N
\end{equation} will be an upper bound for the number of solutions to $f(x , y) = \pm1$ for any binary $n$-ic form~$f$ with absolute discriminant at least $C$. 

Now let $f$ be a binary quartic form and $r$ a positive integer satisfying the conditions of the theorem. Applying Theorem~\ref{maineq} to each $f_{A_j}$, giving $N=\lfloor36+8/\kappa\rfloor$, and then taking the bound in (\ref{p+1N}), now yields Theorem~\ref{gencount}. 
\end{proof}

\vspace{.075in}
\noindent
{\bf Proof of Theorem~\ref{main26}:} 
Let $C>0$ be a constant such that the number of integer solutions to the quartic Thue equation $f(x , y) = \pm 1$ is at most 44 whenever $|\Disc(f)|< C$. We know the latter assumption to be true when $C=10^6$ due to the work of Bennett and Rechnitzer~\cite{BR}, and so we will be particularly interested in this value of $C$. 

For a positive integer $r$, let 
\begin{equation}\label{kappavalue}
\kappa(r):= 1 - \frac{\log D_{1}}{\log(r^{12}C)}-\epsilon
\end{equation}
for a very small $\epsilon>0$ (say $\epsilon=10^{-9}$ or even smaller). 
To obtain the best-possible upper bounds, using Theorem~\ref{gencount}, for the number of integer solutions of quartic Thue equations with no restriction on the size of the discriminant,
we simply need to minimize the function 
\begin{equation}\label{bound}
\psi(r) \lfloor36 + {8}/{\kappa(r)}\rfloor,
\end{equation}
where $r$ ranges over all positive integers with 
\begin{equation}\label{rconstraint}
r^{12}C \geq D_{1}.
\end{equation}
Indeed, for each such $r$, we have $r^{12}C > D_1^{1/(1-\kappa(r))}$ where $0<\kappa(r)<1$, so that Theorem~\ref{gencount} with $\kappa=\kappa(r)$ applies to all quartic Thue equations $f(x,y)=\pm1$ where $|\Disc(f)|\geq C$. 

We list  in Table 1 the optimal value of $r$, and the resulting bound (\ref{bound}), for various values of $C\geq 10^6$.
Since $ \kappa < 1$, we have $36+{8}/{\kappa} > 44$. Therefore, for any integer $r=4$ or $r \geq 6$, we have
$$\psi(r)\lfloor36+{8}/{\kappa}\rfloor \geq 352.
$$
This shows that for $C\geq 10^6$, we need only determine the value of $r\in\{1,2,3,5\}$ satisfying~\eqref{rconstraint}  that minimizes the value of \eqref{bound}. In particular, when $C=10^6$, of these latter values of $r$, only $r=2,3,5$ satisfy $r^{12}C \geq D_{1}$, and the choice of $r=3$ yields the optimal upper bound $\psi(r)\lfloor36 + {8}/{\kappa(r)}\rfloor=276$. This completes the proof of Theorem~\ref{main26}. {$\Box$ \vspace{2 ex}}

Due to expected advances in computation, we have also included in Table 1 the upper bounds for the number of integer solutions of quartic Thue equations that would be implied by Theorem \ref{maineq} if we knew that any quartic Thue equation $f(x , y) = \pm 1$, with 
$\left|\Disc(f)\right| < 10^k$ ($k>6$),
has at most $44$ integer solutions. 


\begin{table}[t]
\begin{center}
\begin{tabular}{ |c|c|c|c|c|c|c|} 
 \hline
$k$ &$r$ & $\kappa(r)$ & $\psi(r)\lfloor36+{8}/{\kappa(r)}\rfloor$\\
\hline\hline
$6$&$3$ &  $0.237$ & $276$ \\
\hline
$7$ & 3&  $0.297$ & $248$ \\
\hline
$8$ & $2$ & $0.230$ & $210$\\
\hline
$9$ & $2$ & $0.291$ & $189$\\
\hline
$10$ & $1$ & $0.106$ & $111$\\
\hline
$11$ & $1$ & $0.187$ & $78$\\
\hline
$12$ & $1$ & $0.255$ & $67$\\
\hline
$13$ & $1$ & $0.312$ & $61$\\
\hline
$14$ & $1$ & $0.361$ & $58$\\
\hline
$15$ & $1$ & $0.404$ & $55$\\
\hline
$20$ & $1$ & $0.553$ & $50$\\
\hline
$30$ & $1$ & $0.702$ & $47$\\
\hline
$45$ & $1$ & $0.801$ & $45$\\
\hline
$>80$ & $1$ &$>0.889$ & $44$\\
\hline
\end{tabular}
\caption{Optimal bounds for $C=10^k$.}
\end{center}
\vspace{-.2in}
\end{table}

These bounds are obtained from Theorem~\ref{gencount} as follows. 
For $k = 7$, we check $r = 2,3$, and we see that $r=3$ yields the better upper bound of~$248$. Since, for any $0< \kappa < 1$ and prime $r \geq 4$,  we have 
$$\psi(r)\lfloor36+{8}/{\kappa}\rfloor > 248,
$$
we need not consider larger values of $r$. Similarly, for $k = 8$, we need only check the values $r =2, 3$, and this time $r=2$ yields the better upper bound of $210$.  For $k=9$, we need to check $r=1$ in addition to $r= 2, 3$, since $r^{12}10^9\geq D_1$ also for $r=1$; we find that $r=2$ yields the best upper bound of 189. Finally, once $k\geq 10$, we see that the choice $r=1$ yields an upper bound
that is less than $3\times 44$, and hence no other $r$ can possibly yield a better bound.  Thus we may take $r=1$ for all $k\geq 10$.

%

If the absolute discriminant of a binary  form is larger than 
$$2.71336712 \times 10^{80},$$
then we obtain the upper bound $44$ already stated in Corollary \ref{corofthm}.

\end{appendices}


\begin{thebibliography}{16}

\bibitem{Akhsmallm}  
S.\ Akhtari,  Representation of small integers by binary forms, {\it The Quarterly Journal of Mathematics: Oxford Journals} {\bf 66} (4) (2015)  1009--1054.

\bibitem{Akhtr} S.~Akhtari,  Representation of unity by binary forms, {\it  Transactions of AMS} (2012) {\bf 364} (2012), 2129--2155.

\bibitem{Akh10}  
S.\ Akhtari, Upper bounds for the number of solutions to quartic Thue equations, {\it International Journal of Number Theory}. (2) {\bf 8} (2012), 335--360. 

\bibitem{A}
S.\ Akhtari, Cubic Thue Equations, {\it Publ.\ Math.\ Debrecen} {\bf 75} (2009), 459--483. 

\bibitem{B}
M.\ A.\ Bennett, On the representation of unity by binary cubic forms, {\it Trans.\ Amer.\ 
Math.\ Soc.} {\bf 353} (2001), 1507--1534.

\bibitem{BR}
M.\ A.\ Bennett and A.~Rechnitzer, Tabulating binary quartic forms over $\mathbb{Z}$ by discriminant. Forthcoming (2021).

\bibitem{BEG}
A.\ B\'erczes, J.-H.\ Evertse, K.\ Gy\H{o}ry,  On the number of equivalence classes of binary forms of given degree and given discriminant, {\it Acta Arith.} {\bf 113} (2004), 363--399. 


\bibitem{hcl1}
M.\ Bhargava, Higher composition laws I: A new view on Gauss composition, and quadratic generalizations,  
{\it Ann.\ of Math.} {\bf 159} (2004), 217--250. 

\bibitem{hcl3}
M.\ Bhargava, Higher composition laws III:
The parametrization of quartic rings, {\it Ann.\ of Math.} {\bf 159} (2004), 
1329--1360.


\bibitem{BST}
M.\ Bhargava, A.\ Shankar, and J.\ Tsimerman,  On the Davenport-Heilbronn theorems and second order terms, {\it Invent.~Math.} {\bf 193} (2013), no.\ 2, 439--499.

\bibitem{BM} B.\ J.\ Birch and J.\ R.\ Merriman, Finiteness theorems
  for binary forms with given discriminant, {\it Proc. London
    Math. Soc. $(3)$} {\bf 24} (1972), 385--394.
    
\bibitem{BoSch} E.~Bombieri and W.~M ~Schmidt, On Thue's equation,  {\it Invent. Math.} {\bf 88} (1987), 69--81.

\bibitem{DF}
B.\ N.\  Delone and D.\ K.\ Faddeev, {\it The theory of irrationalities
of the third degree}, AMS Translations of Mathematical Monographs
{\bf 10}, 1964.

\bibitem{E}
J.-H.\ Evertse, A survey on monogenic orders, {\it Publ.\ Math.\ Debrecen} {\bf 79} (2011), 411--422,

\bibitem{Eve} J.-H.~Evertse, {\it Upper Bounds for the Number of Solutions of Diophantine Equations}, Math.\ Centrum, Amsterdam, 1983.

\bibitem{EG}
J.-H.\ Evertse and K. Gy\H{o}ry, On unit equations and decomposable form equations,
{\it J.\ reine angew.\ Math.} {\bf 358} (1985), 6--19.

\bibitem{EG2}
J.-H.\ Evertse and K. Gy\H{o}ry, {\it Discriminant Equations in Diophantine Number Theory}, Cambridge University Press, 2016.

\bibitem{GPP}
{I. Ga\'al, A. Peth\H{o}} and {M. Pohst}, {Simultaneous representation of numbers by a pair of ternary
quadratic forms---with an application to index form equations in quartic fields,} {\it J.~Number Theory} {\bf 57}
(1996), 90--104.

\bibitem{GGS}
W.-T.\ Gan, B.\ H.\ Gross, and G.\ Savin, 
Fourier coefficients of modular forms on $G_2$, {\it Duke Math. J.}
{\bf 115} (2002), 105--169.

\bibitem{G}
K.\ Gy\H{o}ry, Sur les polyn\^omes \`a coefficients entiers et de discriminant donn\'e III,
{\it Publ.\ Math.\ Debrecen} {\bf 23} (1976), 141--165

\bibitem{Gyo2001} K.~Gy\H{o}ry, Thue inequalities with a small number of primitive solutions, {\it  Period. Math. Hungar.} {\bf 42} (2001), no. 1-2, 199--209.

\bibitem{Kre39} V.~Krechmar. On the superior bound of the number of representation of an integer by binary forms of the fourth degree (Russian),  {\it Bull.\ Acad.\ Sci.\ URSS, Ser.\ Math.} (1939), 289--302.

\bibitem{Levi}
F.\ Levi, Kubische Zahlk\"orper und bin\"are kubische Formenklassen, {\it Leipz.\ Ber.} {\bf 66} (1914), 26--37. 

\bibitem{LM}
D.\ J.\ Lewis and K.\ Mahler, On the representation of integers by binary forms,
{\it Acta Arith.} {\bf 6} (1961), 333--363.

\bibitem{Nakagawa}
J.\ Nakagawa, Binary forms and orders of algebraic number fields, {\it Invent.\ Math.} {\bf 97} (1989), 219--235.

\bibitem{O}
R.\ Okazaki, Geometry of a cubic Thue equation, {\it Publ.\ Math.\ Debrecen} {\bf 61} (2002), 267--314.

\bibitem{SS}
N.\ Saradha and D.\ Sharma, Number of representations of integers by binary forms, {\it Publ.\ Math.\ Debrecen} {\bf 85} (2014), 233--255; Corrigendum, ibid.\ {\bf 86} (2015), 503--504.

\bibitem{Siegel}
C.\ L.\ Siegel, Approximation algebraischer Zahlen, {\it Math.\ Zeit.} {\bf 10} (1921), 173--213.

\bibitem{Ste} C.~L.~Stewart, On the number of solutions of polynomial congruences and Thue equations, {\it J. Amer. Math. Soc. } {\bf 4} (1991), 793--835.

\bibitem{Thue}
A.\ Thue, \"Uber Ann\"aherungswerte algebraischer Zahlen, {\it J.\ reine angew Math.} {\bf 135} (1909), 284--305.

\bibitem{Wood}
M.\ Wood, Quartic rings associated to binary quartic forms, 
{\it Int.\ Math.\ Res.\ Not.} {\bf 2012} (2012), no.\ 6, 1300--1320.

\bibitem{Wood2}
M.\ Wood,
Rings and ideals parametrized by binary $n$-ic forms, {\it J.\ London Math.\ Soc.}~{\bf 83} (2011), 208--231.

\bibitem{YuZh} P.~Yuan and Z.~Zhang, On the number of solutions of Thue equations, {\it AIP Conf. Proc.}~{\bf 1385} (2011), 124-131.


\end{thebibliography}
\end{document}